\documentclass[11pt]{amsart}

\usepackage[T1]{fontenc}
\usepackage[utf8]{inputenc}
\usepackage[english]{babel}
\usepackage{amsmath}
\usepackage{amsfonts}
\usepackage{amssymb}
\usepackage{amsthm}
\usepackage{enumitem}
\usepackage{mymacros}
\usepackage{hyperref}

\newcommand{\sz}{Szeg\H{o}  }

\renewcommand{\MR}[1]{} 

\title{The Smirnov class for spaces with the complete Pick property}

\author[A. Aleman]{Alexandru Aleman}
\address{Lund University, Mathematics, Faculty of Science, P.O. Box 118, S-221 00 Lund, Sweden}
\email{alexandru.aleman@math.lu.se}

\author[M. Hartz]{Michael Hartz}
\address{Department of Mathematics, Washington University in St. Louis, One Brookings Drive,
St. Louis, MO 63130, USA}
\email{mphartz@wustl.edu}
\thanks{M.H. was partially supported by an Ontario Trillium Scholarship and a Feodor Lynen Fellowship}

\author[J. M\raise.5ex\hbox{c}Carthy]{John E. M\raise.5ex\hbox{c}Carthy}
\address{Department of Mathematics, Washington University in St. Louis, One Brookings Drive,
St. Louis, MO 63130, USA}
\email{mccarthy@wustl.edu}
\thanks{J.M. was partially supported by National Science Foundation Grant DMS 1565243}

\author[S. Richter]{Stefan Richter}
\address{Department of Mathematics, University of Tennessee, 1403 Circle Drive, Knoxville, TN 37996-1320, USA}
\email{richter@math.utk.edu}


\begin{document}

\begin{abstract}
  We show that every function in a reproducing kernel Hilbert space with a normalized
  complete Pick kernel is the quotient of a multiplier and a cyclic multiplier.
  This extends a theorem of Alpay, Bolotnikov and Kaptano\u{g}lu.
  We explore various consequences of this result regarding zero sets, spaces on compact sets and
  Gleason parts. In particular, using a construction of Salas, we exhibit a
  rotationally invariant complete Pick space of analytic functions on the unit disc
  for which the corona theorem fails.
\end{abstract}

\maketitle

\section{Introduction}
\label{sec:intro}

The Smirnov class in the unit disc $\bD$ can be characterized as the space
\begin{equation}
  \label{eqa1}
  N^+ \ = \ \{  \varphi/\psi \ : \  \varphi, \psi \in H^\infty, \ \psi \text{ outer} \},
\end{equation}
where $H^\infty$ denotes the algebra of bounded analytic functions on $\bD$.
Alternatively, it is the space of analytic functions on $\bD$ for which the functions
\begin{equation*}
 \theta \mapsto \log ( 1+  \vert f(re^{i\theta}) \vert )
\end{equation*}
are uniformly integrable over $[0,2 \pi]$ for all $0 < r < 1$.
The class occurs frequently in function theory \cite{Duren70,Garnett07,Nikolskiui86}, and has interesting connections to operator theory --- see, e.g. \cite{Helson90,HMN+04}. It contains $H^p$ for all $ p > 0$.

The algebra $H^\infty$ is the multiplier algebra of the Hardy space $H^2$,
which is the  Hilbert function space on $\bD$  which has the \sz kernel
\[
  s(z,w)  = \frac{1}{1- z \ol{w}}
\]
as its reproducing kernel.
Many interesting properties of $H^\infty$ and $H^2$ can be shown to follow from the fact that the \sz kernel is a complete Pick kernel --- see the book \cite{AM02} for a comprehensive account on this topic.

We shall give a definition of complete Pick spaces in Section~\ref{sec:prelim} below. For now,
we simply mention that in addition to the Hardy space $H^2$, the Dirichlet space and the
Drury-Arveson space, which has been studied by many different authors over the years
\cite{Lubin76,Drury78,AP95,DP98,AM00a,Arveson98},
satisfy this property.

Suppose that $\cH$ is a Hilbert
function space on a set $X$ whose kernel $k$ is normalized at $x_0$ (this means that $k(x,x_0) = 1$ for every $x$).
Let $\Mult(\cH)$ denote the multiplier algebra of $\cH$. A multiplier $\psi$ is said to be cyclic if $\psi \cH$
is dense in $\cH$.
We define
\begin{equation*}
  N^+(\cH) = \Big \{ \frac{\varphi}{\psi}: \varphi, \psi \in \Mult(\cH) \text{ and } \psi \text{ is cyclic} \Big\}.
\end{equation*}
(Taking $\cH = H^2$, we recover the classical Smirnov class.)
Observe that a cyclic multiplier does not vanish anywhere, so the quotient $\varphi/\psi$ is defined on all of $X$.
Moreover, since the product of any two cyclic multipliers is cyclic (see Subsection \ref{ss:cyclic}),
$N^+(\cH)$ is an algebra.

Using the inner-outer factorization of functions in $H^2$, one can show that $H^2 \subset N^+$.
On the other hand, there are functions in the Bergman space $L^2_a$
on $\bD$ which do not possess boundary radial
limits anywhere on $\partial \bD$, and hence are not
quotients of two functions in $H^\infty = \Mult(L^2_a)$.
Thus, $L^2_a$ is not contained in $N^+(L^2_a)$.
The following result shows that the inclusion $\cH \subset N^+(\cH)$ is valid
for every complete Pick space $\cH$ with a normalized kernel.
It was shown in the case of the Drury-Arveson space on a finite dimensional ball
by Alpay, Bolotnikov and Kaptano\u{g}lu
\cite[Theorem 10.3]{ABK02}, but their proof can be extended to general complete Pick spaces.

\begin{thm}
  \label{thm:smirnov_intro}
  Let $\cH$ be a
   complete Pick space on $X$ whose kernel is normalized at $x_0 \in X$ and let $f \in \cH$ with $||f||_{\cH} \le 1$.
  Then
  there are $\varphi,\psi \in \Mult(\cH)$ of multiplier norm at most $1$ with $\psi(x_0) = 0$
  such that
  \begin{equation*}
    f = \frac{\varphi}{1 - \psi}.
  \end{equation*}
  In particular, $\cH \subset N^+(\cH)$.
\end{thm}

We will provide a proof of Theorem~\ref{thm:smirnov_intro} in Section~\ref{sec:proof_main}.
This theorem applies in particular to the Dirichlet space,
which answers a question posed in \cite[Section 3]{Ross06} and at the end of \cite{MR15}.

The majority of this note is devoted to exploring consequences of this theorem.
Our first application concerns zero sets.
If $S$ is a class of functions on a set $X$, then a subset $Z$ of $X$ is called
a \emph{zero set for $S$} if there exists a function in $S$ that vanishes exactly on $Z$.
It is well known that the zero sets for
$H^2$ and for $H^\infty$ agree and are precisely the Blaschke
sequences, along with the entire disc $\bD$ (see \cite[Section II.2]{Garnett07}).
Zero sets for functions in $L^2_a$ can be much more complicated, for instance,
the union of two zero sets for $L^2_a$ need not be a zero set for $L^2_a$ \cite{Horowitz74}.
On the other hand, it was shown by Marshall and Sundberg \cite{marsun}, see also \cite[Corollary 9.39]{AM02},
that the zero sets for functions and for multipliers in the Dirichlet space agree.
We show that this fact extends to all complete Pick spaces, thereby providing a positive answer to
\cite[Question 9.28]{AM02} in the case of complete Pick kernels.
The case of Pick kernels which are not necessarily complete Pick kernels remains open.
We say that a reproducing kernel Hilbert space on $X$ is normalized if its kernel is normalized
at some point in $X$.

\begin{cor}
  \label{cor:zero_sets_intro}
  Let $\cH$ be a normalized complete Pick space. Then the zero sets for $\cH$, $\Mult(\cH)$
  and $N^+(\cH)$ agree. In particular, the union of two zero sets for $\cH$ is a zero set for $\cH$.
\end{cor}

If $\cH$ is a Hilbert function space on $X$, then it may happen that
all functions in $\cH$ extend uniquely to a larger set.
For instance, we could start with a space of analytic functions on $\bD$,
let $X = \frac{1}{2} \bD$ and let $\cH$ be the restriction of $\cH$ to $X$,
so that every function in $\cH$ extends uniquely to an analytic function on $\bD$.
The notion of partially multiplicative functional (or generalized kernel function) allows
one to find the largest possible domain of definition for the functions in $\cH$. A non-zero
bounded functional $\rho$ on $\cH$ is said to be \emph{partially multiplicative}
if $\rho(f g) = \rho(f) \rho(g)$ whenever $f,g \in \cH$ such that $f g \in \cH$.
Clearly, evaluation at each point of $X$ gives rise to a partially multiplicative
functional on $\cH$. We say that \emph{$X$ is a maximal domain for $\cH$} if conversely, every partially
multiplicative functional is given by evaluation at a point in $X$ (this is called \emph{$\cH$ is algebraically
consistent} by Cowen and MacCluer, see \cite[Definition 1.5]{CM95}).
It was shown in \cite[Section 3]{MS15} that every Hilbert function space can be considered as a space of functions on the
set of partially multiplicative
functionals, and this set is a maximal domain for $\cH$.
For more discussion, the reader is referred to
\cite[Definition 1.5]{CM95}, \cite[Section 5]{Hartz15} and \cite[Section 3]{MS15}.
In the context of normalized complete Pick spaces, there is another notion of partially multiplicative functional,
and we show that the two notions agree. This answers a question that was left open in \cite{Hartz15}.

\begin{cor}
\label{pra3}
  Let $\cH$ be a normalized complete Pick space and let $\rho$ be a bounded non-zero functional on $\cH$.
  Then the following are equivalent.
  \begin{enumerate}[label=\normalfont{(\roman*)}]
    \item $\rho(\varphi g) = \rho(\varphi) \rho(g)$ for all $\varphi \in \Mult(\cH)$ and all $g \in \cH$.
    \item $\rho(f g) = \rho(f) \rho(g)$ whenever $f,g \in \cH$ such that $f g \in \cH$.
  \end{enumerate}
\end{cor}

A classical notion in uniform algebras is that of a Gleason part in the maximal ideal space,
see \cite{Gleason57} and \cite[Chapter VI]{Gamelin69}. For instance, Gleason parts play an important
role in the study of the maximal ideal space of $H^\infty$, see for example \cite[Chapter X]{Garnett07}.
This notion can be generalized to multiplier algebras.
Given two characters $\rho_1,\rho_2$ on a multiplier
algebra, we write $\rho_1 \sim \rho_2$ if $||\rho_1 - \rho_2|| < 2$. This defines
an equivalence relation on the maximal ideal space of the multiplier algebra (see Lemma \ref{lem:Gleason_parts}),
and the equivalence classes are called \emph{Gleason parts}.
In this context, we prove in Section \ref{sec:gleason} the following result.

\begin{prop}
\label{prop:gleason_introduction}
  Let $\cH$ be a normalized complete Pick space on a set $X$. Then the following are equivalent.
  \begin{enumerate}[label=\normalfont{(\roman*)}]
    \item $X$ is a maximal domain for $\cH$.
    \item Every weak-$*$ continuous character on $\Mult(\cH)$ is given by evaluation at a point in $X$.
    \item The characters of evaluation at points in $X$ form a Gleason part in the maximal
      ideal space of $X$.
  \end{enumerate}
\end{prop}

The equivalence of (i) and (ii) in Proposition \ref{prop:gleason_introduction} was already observed in \cite[Proposition 5.6]{Hartz15} (using
(i) in Corollary \ref{pra3} as the definition of partially multiplicative functional and under the assumption
that $\cH$ separates the points in $X$).

In Section~\ref{sec:compact}
we study complete Pick spaces of continuous functions on compact sets --- for example, spaces
of analytic functions on the disc that extend to be continuous on $\overline{\bD}$.
In particular, we investigate the validity of a corona theorem in this context.
Carleson's famous corona theorem for $H^\infty$ \cite{Carleson62} asserts that the open unit disc is dense
in the maximal ideal space of $H^\infty$. This was extended to multiplier algebras
of certain Besov-Sobolev spaces on the unit ball in finite dimensions,
including the multiplier algebra of the Drury-Arveson space, by Costea, Sawyer and Wick \cite{CSW11}.
If $\Mult(\cH)$ consists of continuous functions on a compact set $X$, then
a corona theorem for $\Mult(\cH)$ would assert that the maximal ideal space of $\Mult(\cH)$
is equal to the characters of evaluation at points of $X$.

\begin{prop}
  \label{pra4}
  Let $\cH$ be a
  normalized complete Pick space of continuous functions
  on a compact set $X$ such that $X$ is a maximal domain for $\cH$ and
  such that $\Mult(\cH)$ separates the points of $X$. Then the following are equivalent.
  \begin{enumerate}[label=\normalfont{(\roman*)}]
    \item $\Mult(\cH) = \cH$ as vector spaces.
    \item The corona theorem holds for $\Mult(\cH)$, that is,
      the maximal ideal space of $\Mult(\cH)$
      is $X$.
    \item The one-function corona theorem holds for $\Mult(\cH)$, that is, if
      $\varphi \in \Mult(\cH)$ is non-vanishing, then $1/ \varphi \in \Mult(\cH)$.
  \end{enumerate}
\end{prop}

As a consequence, using a very interesting example of Salas \cite{Salas81}, we exhibit a rotationally invariant
complete Pick space on $\overline{\bD}$ for which the one-function corona theorem (and hence the full corona theorem) fails.

\begin{thm}
  \label{thm:corona_intro}
  There exists a complete Pick space $\cH$ on $\ol{\bD}$ such that $\ol{\bD}$
  is a maximal domain for $\cH$ with a reproducing kernel
  of the form
  \begin{equation*}
    k(z,w) = \sum_{n=0}^\infty a_n (z \ol{w})^n
  \end{equation*}
  and such that the one-function corona theorem for $\Mult(\cH)$ fails. In particular, $\ol{\bD}$
  is a proper compact subset of the maximal ideal space of $\Mult(\cH)$.
\end{thm}

In the terminology of Nikolski, the fact that the one-function corona theorem fails for $\Mult(\cH)$
means that the spectrum of $\Mult(\cH)$ is not \emph{1-visible}, see Definition 0.2.1 and Lemma 0.2.2
in \cite{Nikolski99}. We also remark that an example of a multiplier algebra
of functions on a subset of $\bC$ for which the one-function corona theorem fails was already
obtained by Trent \cite{Trent08}, but the multiplier algebra of Theorem \ref{thm:corona_intro} is quite different.
Indeed, Trent remarks that the functions in his multiplier algebra
``have no smoothness properties in general'' (see the discussion preceding \cite[Lemma 1]{Trent08}),
whereas the multipliers of Theorem \ref{thm:corona_intro} are analytic functions in the unit disc.

\section{Preliminaries}
\label{sec:prelim}

\subsection{Kernels, multipliers and normalization}
Let $\cH$ be a reproducing kernel Hilbert space on a set $X$ with reproducing kernel $k$.
For background material on reproducing kernel Hilbert spaces, the reader is referred to
\cite{AM02} and \cite{PR16}. We will assume throughout that $k$ does not vanish on the diagonal.
Let
\begin{equation*}
  \Mult(\cH) = \{ \varphi: X \to \bC: \varphi f \in \cH \text{ for all } f \in \cH \}
\end{equation*}
denote the \emph{multiplier algebra} of $\cH$. The closed graph theorem shows that
every $\varphi \in \Mult(\cH)$ determines a bounded multiplication operator $M_\varphi$ on $\cH$,
and we set $||\varphi||_{\Mult(\cH)} = ||M_\varphi||$.
We say that $k$ (or $\cH$) is \emph{normalized at $x_0 \in X$}
if $k(x,x_0) = 1$ for all $x \in X$, and we say that $k$ (or $\cH$) is \emph{normalized}
if it is normalized at some point in $X$.
If $k(z,w) \neq 0$ for all $z,w \in X$, then it is always
possible to normalize the kernel at a point (see \cite[Section 2.6]{AM02}). This procedure multiplies
all functions in $\cH$ by a fixed non-vanishing function and leaves $\Mult(\cH)$ unchanged.
If $k$ is normalized, then $\cH$ contains in particular the constant function $1$ and
$||1||_{\cH} = 1$, so that
$\Mult(\cH) \subset \cH$, and the inclusion is contractive.

\subsection{The Pick property}
We say that $\cH$ is a \emph{complete Pick space}
if for every $r \in \bN$ and every finite collection of points $z_1,\ldots,z_n \in X$
and matrices $W_1,\ldots,W_n \in M_r(\bC)$,
positivity of the $nr \times nr$-block matrix
\begin{equation*}
  \Big[ k(z_i,z_j) (I_{\bC^r} - W_i W_j^*) \Big]_{i,j=1}^n
\end{equation*}
implies that there exists $\Phi \in M_r(\Mult(\cH))$ of norm at most $1$ such that
\begin{equation*}
  \Phi(z_i) = W_i \quad (i=1,\ldots,n).
\end{equation*}
In this setting, we also say that $k$ is a \emph{complete Pick kernel}.
If $\cH$ is a normalized complete Pick space, then $k(z,w) \neq 0$ for all $z,w \in X$
by \cite[Lemma 7.2]{AM02}.
Complete Pick spaces were characterized
by a theorem of Quiggin \cite{Quiggin93} and McCullough \cite{McCullough92}. We require the following
characterization of Agler and McCarthy \cite{AM00}.

\begin{thm}[Agler-McCarthy]
  \label{thm:AM}
  Let $\cH$ be a reproducing kernel Hilbert space on $X$ with kernel $k$ which is normalized at $x_0 \in X$.
  Then $\cH$ is a complete Pick space if and only if there exists an auxiliary Hilbert space $\cE$
  and a function $b: X \to \cE$ with $||b(w)|| < 1$ for all $w \in X$, $b(x_0) = 0$ and
  \begin{equation*}
    k(z,w) = \frac{1}{1 - \langle b(z),b(w) \rangle_{\cE}}.
  \end{equation*}
  If $\cH$ is separable, then $\cE$ can be chosen to be separable.
\end{thm}

A comprehensive treatment of complete Pick spaces, as well as examples besides the ones already mentioned
in the introduction, can be found in \cite{AM02}.

\subsection{Cyclic multipliers}
\label{ss:cyclic}

We say that a multiplier $\varphi$ of a reproducing kernel Hilbert space $\cH$
is \emph{cyclic} if
the multiplication operator $M_\varphi$ on $\cH$ has dense range.
Observe that $\varphi$ is cyclic if and only if
$\ker(M_\varphi^*) = \{0\}$.
In particular, the product of two cyclic multipliers is cyclic.

We require the following version of the maximum modulus principle
for Hilbert function spaces with non-vanishing kernels.
\begin{lem}
  \label{lem:max_mod}
  Let $\cH$ be a reproducing kernel Hilbert space on $X$ with kernel $k$ such that $k(z,w) \neq 0$ for all
  $z,w \in X$ and let $\varphi \in \Mult(\cH)$ with $||\varphi||_{\Mult(\cH)} \le 1$.
  If there exists $z \in X$ with $|\varphi(z)| = 1$, then $\varphi$ is constant.
\end{lem}

\begin{proof}
  By multiplying $\varphi$ by a complex number number of modulus $1$, we may assume that $\varphi(z) = 1$. Let
  $w \in X$ be arbitrary. Since $||\varphi||_{\Mult(\cH)} \le 1$, the Pick matrix at $\{z,w\}$, which is
  \begin{align*}
    \begin{pmatrix}
      0 & k(z,w) (1 - \ol{\varphi(w)}) \\
      k(w,z) ( 1 - \varphi(w)) & k(y,y) ( 1 - |\varphi(w)|^2)
    \end{pmatrix},
  \end{align*}
  is positive semidefinite. Taking its determinant, we see that
  \begin{equation*}
    0 \le - |k(z,w)|^2  |1 - \varphi(w)|^2.
  \end{equation*}
  Since $k(z,w) \neq 0$ by assumption, this inequality implies that $\varphi(w) = 1$.
  Consequently, $\varphi$ is identically $1$.
\end{proof}

The following lemma gives a sufficient condition for cyclicity of multipliers, which is probably known.

\begin{lem}
  \label{lem:denominator_cyclic}
  Let $\cH$ be a reproducing kernel Hilbert space on $X$ with kernel $k$ such that $k(z,w) \neq 0$
  for all $z,w \in X$.
  Let $\varphi \in \Mult(\cH)$ with $||\varphi||_{\Mult(\cH)} \le 1$.
  If $\varphi \neq 1$, then $1 - \varphi$ is cyclic.
\end{lem}

\begin{proof}
  It suffices to show that $M_{1 - \varphi}^* = 1 - M_\varphi^*$ is injective.
  Thus, let $h \in \ker( 1 - M_{\varphi}^*)$. Then $M_\varphi^* h = h$. Since $M_\varphi$ is a contraction,
  it follows that $M_\varphi h = h$ and hence that
  $(\varphi - 1) h = 0$.
  Since $\varphi \neq 1$, Lemma \ref{lem:max_mod} implies that $\varphi - 1$ has no zeros on $X$. Consequently,
  $h = 0$, so that $1 - \varphi$ is cyclic.
\end{proof}

\section{Proof of Theorem \ref{thm:smirnov_intro} and first consequences}
\label{sec:proof_main}

We now present the proof of Theorem \ref{thm:smirnov_intro}. For the convenience of the reader,
we restate the result.
\begin{thm}
  \label{thm:smirnov}
  Let $\cH$ be a complete Pick space on $X$ whose kernel is normalized at $x_0 \in X$ and let $f \in \cH$ with $||f||_{\cH} \le 1$.
  Then
  there are $\varphi,\psi \in \Mult(\cH)$ of multiplier norm at most $1$ with $\psi(x_0) = 0$
  such that
  \begin{equation*}
    f = \frac{\varphi}{1 - \psi}.
  \end{equation*}
  In particular, $\cH \subset N^+(\cH)$.
\end{thm}

\begin{proof}
  By Theorem \ref{thm:AM}, there exists
  a Hilbert space $\cE$ and a function $b: X \to \cB(\cE,\bC)$ with $b(x_0) = 0$ such that
  \begin{equation*}
    k(z,w) = \frac{1}{1- b(z) b(w)^* }.
  \end{equation*}
  Note that $b \in \Mult(\cH \otimes \cE, \cH)$ has multiplier norm at most $1$,
  as $k(z,w) (1 - b(z) b(w)^*) = 1$ is positive definite.
  Define $\Phi: X \to \cB( \bC \oplus \cE,\bC)$ by
  \begin{equation*}
    \Phi(z) = (1, f(z) b(z)).
  \end{equation*}
  Then
  \begin{align*}
    k(z,w) \Phi(z) \Phi(w)^* &= k(z,w) + f(z) \ol{f(w)} k(z,w) b(z) b(w)^* \\
    &= k(z,w) + f(z) \ol{f(w)} k(z,w) - f(z) \ol{f(w)}
  \end{align*}
  and consequently
  \begin{equation*}
    k(z,w) ( \Phi(z) \Phi(w)^* - f(z) \ol{f(w)}) = k(z,w) - f(z) \ol{f(w)}.
  \end{equation*}
  Since $||f||_{\cH} \le 1$ this function is positive definite (see \cite[Theorem 3.11]{PR16}). In this setting,
  a version of Leech's theorem \cite{BTV01,AT02}, or more precisely the implication
  (i) $\Rightarrow$ (ii) of \cite[Theorem 8.57]{AM02}, shows that there exists
  a vector valued multiplier $\Psi \in \Mult(\cH, \cH \otimes (\bC \oplus \cE))$ of norm
  at most $1$ such that
  \begin{equation*}
    \Phi(z) \Psi(z) = f(z) \quad (z \in X).
  \end{equation*}
  In the statement of \cite[Theorem 8.57]{AM02}, it is assumed that $k$ satisfies an irreducibility
  condition that implies that $\cH$ separates the points of $X$, but the proof shows that it suffices to assume that $k$ is normalized.
  Write
  \begin{equation*}
    \Psi(z) =
    \begin{pmatrix}
      \varphi(z) \\ \widetilde{\Psi}(z)
    \end{pmatrix},
  \end{equation*}
  where $\varphi \in \Mult(\cH)$ and $\widetilde{\Psi}(z) \in \Mult(\cH, \cH \otimes \cE)$
  both have norm at most one.
  Then
  \begin{equation*}
    \varphi(z) + f(z) b(z) \widetilde{\Psi}(z) = f(z),
  \end{equation*}
  so that
  \begin{equation*}
    f(z) ( 1 - b(z) \widetilde{\Psi}(z)) = \varphi(z).
  \end{equation*}
  Defining $\psi(z) = b(z) \widetilde{\Psi}(z)$, we obtain the desired representation of $f$.
  Lemma \ref{lem:denominator_cyclic} implies that $1 - \psi$ is a cyclic multiplier, which shows
  that $f \in N^+(\cH)$ and hence proves the additional assertion.
\end{proof}

Corollary \ref{cor:zero_sets_intro} is now an immediate consequence of the preceding theorem.

\begin{cor}
  \label{cor:zero_sets}
  Let $\cH$ be a normalized complete Pick space.
  Then the zero sets for $\cH$, $\Mult(\cH)$
  and $N^+(\cH)$ agree. In particular, the union of two zero sets for $\cH$ is a zero set for $\cH$.
\end{cor}

\begin{proof}
  Observe that
  \begin{equation*}
    \Mult(\cH) \subset \cH \subset N^+(\cH),
  \end{equation*}
  where the first inclusion holds since $\cH$ is normalized, so that $1 \in \cH$,
  and the second inclusion follows from Theorem \ref{thm:smirnov}.
  It is immediate from the definition of $N^+(\cH)$
  that every zero set for $N^+(\cH)$ is a zero set for $\Mult(\cH)$,
  so that the zero sets for all three spaces agree. Since $\Mult(\cH)$
  is an algebra, the union of two zero sets is a zero set.
\end{proof}

\begin{rem}
  To obtain equality of the zero sets for $\cH$ and for $\Mult(\cH)$, it suffices
  to assume that $\cH$ is a complete Pick space whose kernel does not vanish anywhere.
  Indeed, in this case, the kernel can be normalized at a point (see \cite[Section 2.6]{AM02}),
  which does not affect the zero sets.
\end{rem}

It is also not hard to deduce Corollary \ref{pra3} from Theorem \ref{thm:smirnov}.

\begin{cor}
  \label{cor:part_mult}
  Let $\cH$ be a normalized complete Pick space and let $\rho$ be a bounded non-zero functional on $\cH$.
  Then the following are equivalent.
  \begin{enumerate}[label=\normalfont{(\roman*)}]
    \item $\rho(\varphi g) = \rho(\varphi) \rho(g)$ for all $\varphi \in \Mult(\cH)$ and all $g \in \cH$.
    \item $\rho(f g) = \rho(f) \rho(g)$ whenever $f,g \in \cH$ such that $f g \in \cH$.
  \end{enumerate}
\end{cor}

\begin{proof}
  The implication (ii) $\Rightarrow$ (i) is trivial since $\Mult(\cH) \subset \cH$.
  Conversely, assume that (i) holds  and let $f,g \in \cH$
  such that $f g \in \cH$. By Theorem \ref{thm:smirnov}, we may write $f = \varphi /\psi$, where $\varphi , \psi \in \Mult(\cH)$
  and $\psi$ is cyclic.
  We claim that $\rho(\psi) \neq 0$.
  Indeed, if $\rho(\psi) = 0$, then
  $\rho(\psi g) = 0$ for all $g \in \cH$ and hence $\rho = 0$, as $\psi$ is cyclic.
  Therefore, $\rho(\psi) \neq 0$.
  Moreover,
  \begin{equation*}
    \rho(\psi) \rho(f g) = \rho(\psi f g) = \rho( \varphi g) = \rho(\varphi) \rho(g) = \rho(\psi f) \rho(g)
    = \rho(\psi) \rho(f) \rho(g).
  \end{equation*}
  Since $\rho(\psi) \neq 0$, it follows that $\rho(f g) = \rho(f) \rho(g)$.
\end{proof}

\section{Gleason parts}
\label{sec:gleason}

Gleason parts were originally introduced in the context of uniform algebras,
but it is possible to generalize this notion to multiplier algebras of
reproducing kernel Hilbert spaces, and indeed to arbitrary unital operator algebras of functions.
This was observed by Rochberg \cite[Proposition 8]{Rochberg14} for multiplier algebras of
certain weighted Dirichlet spaces on $\bD$,
but his arguments readily generalize. For completeness, we provide the proof of the lemma below.
It is an easy adaptation of \cite[Theorem VI.2.1]{Gamelin69}.

\begin{lem}
  \label{lem:Gleason_parts}
  Let $\cH$ be a reproducing kernel Hilbert space and let $\rho_1,\rho_2$ be characters
  on $\Mult(\cH)$. Then the following are equivalent.
  \begin{enumerate}[label=\normalfont{(\roman*)}]
    \item $||\rho_1 - \rho_2 || < 2$.
    \item $|| \rho_1 \big|_{\ker(\rho_2)}|| < 1$.
    \item Whenever $(\varphi_n)_n$ is a sequence in the unit ball of $\Mult(\cH)$ such that
      $\lim_{n \to \infty} |\rho_1(\varphi_n)| = 1$, then $\lim_{n \to \infty} |\rho_2(\varphi_n)| = 1$.
  \end{enumerate}
  In particular, the relation $\rho_1 \sim \rho_2$ if and only if $||\rho_1 - \rho_2|| < 2$
  is an equivalence relation.
\end{lem}

\begin{proof}
  For $a \in \bD$, let $\theta_a$ denote the conformal automorphism of $\bD$ defined by
  \begin{equation*}
    \theta_a(z) = \frac{a - z}{1 - \ol{a} z}.
  \end{equation*}
  Thus, $\theta_a$ is an involution which interchanges $0$ and $a$.
  We will make repeated use of the following consequence of von Neumann's inequality:
  if $\varphi \in \Mult(\cH)$ with $||\varphi||_{\Mult(\cH)} \le 1$, then
  $\theta_a \circ \varphi \in \Mult(\cH)$
  with $||\theta_a \circ \varphi||_{\Mult(\cH)} \le 1$ for all $a \in \bD$.

  (i) $\Rightarrow$ (ii) Suppose that $||\rho_1 \big|_{\ker(\rho_2)}|| =1$. Then
  there exists a sequence $(\varphi_n)$ in the unit ball of $\ker(\rho_2)$ such that if $t_n = \rho_1(\varphi_n)$, then
  $0 < t_n < 1$ for all $n \in \bN$ and $\lim_{n \to \infty} t_n = 1$. Let $a_n = (1 - \sqrt{1 - t_n^2})/t_n$. It is straightforward
  to check that $a_n \in (0,1)$ and that $\theta_{a_n}(t_n) = - \theta_{a_n}(0)$ for all $n \in \bN$. Thus,
  \begin{align*}
    ||\rho_1 - \rho_2|| &\ge |\rho_1 (\theta_{a_n} \circ \varphi_n) - \rho_2( \theta_{a_n} \circ \varphi_n)| =
    |\theta_{a_n} (t_n) - \theta_{a_n} (0)| \\
    &= 2 |a_n| \xrightarrow{n \to \infty} 2.
  \end{align*}
  Consequently, $||\rho_1 - \rho_2|| =2$.

  (ii) $\Rightarrow$ (iii) Suppose that there exists a sequence $(\varphi_n)$ in the unit ball
  of $\Mult(\cH)$ such that $|\rho_1(\varphi_n)|$ tends to $1$, but $a_n = \rho_2(\varphi_n)$ is bounded away
  from $1$ in modulus. Then $\theta_{a_n} \circ \varphi_n$ belongs to the unit ball of $\ker(\rho_2)$
  and it is easy to see that $|\rho_1(\theta_{a_n} \circ \varphi_n)|$ tends to $1$. So $||\rho_1 \big|_{\ker(\rho_2)}|| = 1$.

  (iii) $\Rightarrow$ (ii) Suppose that $||\rho_1 \big|_{\ker(\rho_2)}|| = 1$. Then there exists a sequence
  $(\varphi_n)$ in the unit ball of $\ker(\rho_2)$ such that $\lim_{n \to \infty} |\rho_1(\varphi_n)| = 1$.
  Therefore, (iii) fails.

  (ii) $\Rightarrow$ (i) Suppose that $||\rho_1 - \rho_2|| = 2$. Then there exists a sequence $(\varphi_n)$ in the unit
  ball of $\Mult(\cH)$ such that $a_n = \rho_1(\varphi_n)$ and $b_n = \rho_2(\varphi_n)$ both belong to $\bD$ and $|a_n - b_n|$
  tends to $2$. Then $\theta_{b_n} \circ \varphi_n$ belongs to the unit ball of $\ker(\rho_2)$ and
  \begin{equation*}
    |\rho_1(\theta_{b_n} \circ \varphi_n)| = |\theta_{b_n} (a_n)| \ge \frac{|a_n - b_n|}{2} \xrightarrow{n \to \infty} 1.
  \end{equation*}
  Thus, $||\rho_1 \big|_{\ker(\rho_2)}|| = 1$.

  Finally, it follows from (i) that $\sim$ is reflexive and symmetric, and it follows from (iii) that $\sim$ is transitive.
\end{proof}

Let $\cH$ be a reproducing kernel Hilbert space on $X$.
Then the assignment $\varphi \mapsto M_\varphi$ identifies $\Mult(\cH)$
with a weak-$*$ closed subalgebra of $B(\cH)$, hence $\Mult(\cH)$ is a dual space in its own right,
and we endow it with this weak-$*$ topology.
It is straightforward to check that on bounded subsets of $\Mult(\cH)$, the weak-$*$ topology
coincides with the topology of pointwise convergence on $X$.
We are now in the position to prove a more detailed version of Proposition \ref{prop:gleason_introduction}.

\begin{prop}
  \label{prop:gleason}
  Let $\cH$ be a complete Pick space on a set $X$ whose kernel is normalized at $x_0 \in X$
  and let $\delta_0$ denote the character of evaluation at $x_0$.
  For a character $\rho$ on $\Mult(\cH)$, the following assertions are equivalent.
  \begin{enumerate}[label=\normalfont{(\roman*)}]
    \item $\rho$ extends to a bounded partially multiplicative functional on $\cH$.
    \item $\rho$ is weak-$*$ continuous.
    \item $\rho$ belongs to the Gleason part of $\delta_0$.
  \end{enumerate}
  It follows that the following are equivalent.
  \begin{enumerate}[label=\normalfont{(\roman*')}]
    \item $X$ is a maximal domain for $\cH$.
    \item Every weak-$*$ continuous character on $\Mult(\cH)$ is given by evaluation at a point in $X$.
    \item The characters of evaluation at points in $X$ form a Gleason part in the maximal
      ideal space of $X$.
  \end{enumerate}
\end{prop}

\begin{proof}
  (i) $\Rightarrow$ (ii) If $\widetilde \rho$ is a bounded partially multiplicative functional
  on $\cH$ which extends $\rho$, then $\rho(\varphi) = \widetilde \rho( M_\varphi 1)$ for $\varphi \in \Mult(\cH)$,
  which shows that $\rho$ is WOT continuous and thus weak-$*$ continuous.

  (ii) $\Rightarrow$ (iii) Suppose that $\rho$ is weak-$*$ continuous and assume
  toward a contradiction that $||\rho|_{\ker(\delta_0)}|| = 1$.
  Since $\delta_0$ is weak-$*$ continuous,
  the unit ball of $\ker(\delta_0)$ is weak-$*$ compact, so there exists a multiplier
  $\psi$ of norm $1$ such that $\psi(x_0) = 0$ and $\rho(\psi)=1$. If follows
  from Lemma \ref{lem:max_mod} that $|\psi(x)| < 1$ for all $x \in X$, so the sequence
  $(\psi^n)$ converges to zero pointwise, and hence in the weak-$*$ topology.
  But $\rho(\psi^n) = 1$ for all $n \in \bN$, contradicting the fact that $\rho$ is weak-$*$ continuous.
  Therefore, $||\rho|_{\ker(\delta_0)}|| < 1$, so (iii) holds by Lemma \ref{lem:Gleason_parts}.

  (iii) $\Rightarrow$ (i) Suppose that $\rho$ belongs to the Gleason part of $\delta_0$,
  so that $\alpha = ||\rho|_{\ker(\delta_0)}|| < 1$ by Lemma \ref{lem:Gleason_parts}.
  If $f \in \cH$ has a representation $f = \varphi_1 / \varphi_2$, where $\varphi_1$ and $\varphi_2$
  are multipliers and $\rho(\varphi_2) \neq 0$, we define
  \begin{equation*}
    \widetilde \rho(f) = \frac{\rho(\varphi_1)}{\rho(\varphi_2)}.
  \end{equation*}
  By Theorem \ref{thm:smirnov}, every $f \in \cH$ can be written
  as $f = \varphi / (1 - \psi)$ with $||\varphi||_{\Mult(\cH)} \le ||f||_{\cH}$ and
  $\psi$ belonging to the unit ball of $\ker(\delta_0)$.
  Since $|\rho(\psi)| \le \alpha < 1$ by assumption, we see in particular that every $f \in \cH$
  admits a representation as above.
  Since $\rho$ is a character, $\widetilde \rho$ is well defined, linear and satisfies
  $\widetilde \rho(\varphi f)
  = \widetilde \rho(\varphi) \widetilde \rho(f)$ for $\varphi \in \Mult(\cH)$ and $f \in \cH$.
  Moreover, writing $f = \varphi/(1 - \psi)$ as above,
  we obtain the estimate
  \begin{equation*}
    |\widetilde \rho(f)| \le \frac{|\rho(\varphi)|}{1 - |\rho(\psi)|}
    \le \frac{1}{1 - \alpha} ||f||_{\cH},
  \end{equation*}
  so that $\widetilde \rho$ is a bounded functional that extends $\rho$, and $\widetilde \rho$
  is partially multiplicative by Corollary \ref{cor:part_mult}.

  Finally, to deduce the equivalence of (i'), (ii') and (iii') from the equivalence of (i), (ii) and (iii),
  it only remains to show that a bounded partially multiplicative functional on $\cH$
  is uniquely determined by its values on $\Mult(\cH)$. This can be seen as in the proof of Corollary
  \ref{cor:part_mult}, or alternatively, if follows from the fact that
  $\Mult(\cH)$ is dense in $\cH$.
\end{proof}

\begin{rem}
  Suppose that $\cH$ is a normalized complete Pick space on $X$ and assume for convenience
  that $\cH$ separates the points of $X$. If $X$ is not a maximal
  domain for $\cH$, then Proposition \ref{prop:gleason} provides three equivalent ways
  of enlarging $X$ to a maximal domain for $\cH$. There is a fourth equivalent way, which we will now
  briefly describe.

  For a cardinal number $d$, let $\bB_d$ denote the open unit ball
  $\bB_d$ in a Hilbert space of dimension $d$. The Drury-Arveson space $H^2_d$ is the reproducing kernel
  Hilbert space on $\bB_d$ with kernel
  \begin{equation*}
    \frac{1}{1 - \langle z,w \rangle}.
  \end{equation*}
  Let $b: X \to \bB_d$ be the map of Theorem \ref{thm:AM} and let $S = b(X)$. Then
  \begin{equation*}
    H^2_d \big|_S \to \cH, \quad f \mapsto f \circ b,
  \end{equation*}
  is a unitary operator. Let
  \begin{equation*}
    I = \{ \varphi \in \Mult(H^2_d): \varphi \big|_S = 0\}
  \end{equation*}
  and let
  \begin{equation*}
    V = \{z \in \bB_d: \varphi(z) = 0 \text{ for all } \varphi \in I \}.
  \end{equation*}
  ($V$ is an analogue of the Zariski closure from algebraic geometry.)
  Tautologically, $S \subset V$, and it was observed by Davidson, Ramsey and Shalit \cite[Section 2]{DRS15}
  that $H^2_d \big|_S$ can be identified with $H^2_d \big|_V$. More precisely,
  every function in $H^2_d \big|_V$ extends uniquely to a function in $H^2_d \big|_S$.
  Moreover, \cite[Lemma 5.4]{Hartz15}
  and its proof show that the partially multiplicative functionals on $H^2_d \big|_S$ (and hence on $\cH$)
  precisely correspond to points in $V$.

  In particular, this last description of a maximal domain for $\cH$ shows that $\cH$ remains
  a complete Pick space after enlarging the domain $X$ to a maximal one.
\end{rem}

\section{Spaces on compact sets and the corona theorem}
\label{sec:compact}

In this section, we study spaces of continuous functions on compact sets and prove
Proposition \ref{pra4} and Theorem \ref{thm:corona_intro}. We begin with a few preliminaries about
corona theorems.

Let $\cH$ be a reproducing kernel Hilbert space on a set $X$ such that $\Mult(\cH)$ separates the points of $X$
and let $\cM(\Mult(\cH))$ denote the maximal ideal space of $\Mult(\cH)$.
Then $X$ can be identified with a subset of $\cM(\Mult(\cH))$ via point evaluations.
We say that \emph{the corona theorem holds for $\Mult(\cH)$}
if $X$ is dense in $\cM(\Mult(\cH))$ in the Gelfand topology.

Gelfand theory shows that the following two statements are equivalent.
\begin{enumerate}[label=\normalfont{(\roman*)}]
  \item
The corona theorem holds for $\Mult(\cH)$.
\item
If $\varphi_1,\ldots,\varphi_n \in \Mult(\cH)$ such that
\begin{equation*}
  \inf_{x \in X} \sum_{j=1}^n |\varphi_j(x)| > 0,
\end{equation*}
then the ideal generated by $\varphi_1,\ldots,\varphi_n$ inside $\Mult(\cH)$ is all of $\Mult(\cH)$.
\end{enumerate}
We say that the \emph{one-function corona theorem holds for $\Mult(\cH)$}
if whenever $\varphi \in \Mult(\cH)$ and $\inf_{x \in X} |\varphi(x)| > 0$, then $1/\varphi \in \Mult(\cH)$.
Thus, the corona theorem for $\Mult(\cH)$ implies the one-function corona theorem for $\Mult(\cH)$.

\begin{rem}
  It is not hard to see that in the setting above, the one-function corona
  theorem holds for $\Mult(\cH)$ if and only if for every $\varphi \in \Mult(\cH)$,
  we have
  \begin{equation*}
    \sigma_{\Mult(\cH)} (\varphi) = \ol{ \{ \varphi(x): x \in X \}}.
  \end{equation*}
  Here, $\sigma_{\Mult(\cH)}$ denotes the spectrum in the unital Banach algebra $\Mult(\cH)$.
  Similarly, the corona theorem holds for $\Mult(\cH)$ if and only if for every $n \in \bN$,
  we have
  \begin{equation*}
    \sigma_{\Mult(\cH)} (\varphi_1,\ldots,\varphi_n) = \ol{ \{ (\varphi_1(x), \ldots, \varphi_n(x)): x \in X  \} }
    \subset \bC^n.
  \end{equation*}
  If $\cH$ is a normalized complete Pick space, then
  we may replace $\sigma_{\Mult(\cH)}$ with other notions of joint spectrum. Indeed,
  the Toeplitz corona theorem
  (see \cite[Section 8.4]{AM02}) shows that in this case,
  \begin{equation*}
    \sigma_{\Mult(\cH)} (\varphi_1,\ldots,\varphi_n) = \sigma_r (M_{\varphi_1},\ldots,M_{\varphi_n})
    = \sigma_T(M_{\varphi_1},\ldots,M_{\varphi_n}),
  \end{equation*}
  where $\sigma_r$ and $\sigma_T$ denote the right spectrum and the Taylor spectrum, respectively
  (see \cite{Mueller07} for a definition and discussion of these notions).
\end{rem}

For normalized complete Pick spaces, it is easy to tell from the kernel whether the multiplier algebra
separates the points of the underlying set.

\begin{lem}
  \label{lem:Pick_separate_points}
  Let $\cH$ be a normalized complete Pick space on $X$ with kernel $k$. Then the following are equivalent.
  \begin{enumerate}[label=\normalfont{(\roman*)}]
    \item $\Mult(\cH)$ separates the points of $X$.
    \item $\cH$ separates the points of $X$.
    \item If $z \neq w$, then $k(\cdot,z) \neq k(\cdot,w)$.
    \item If $z \neq w$, then $k(\cdot,z)$ and  $k(\cdot,w)$ are linearly independent.
  \end{enumerate}
\end{lem}

\begin{proof}
  (i) $\Rightarrow$ (ii) $\Rightarrow$ (iii) is trivial, and (iii) $\Rightarrow$ (iv) follows
  from the fact that $k$ is normalized at a point.

  Finally, if (iv) holds, then the Cauchy-Schwarz inequality and the Pick property, applied to the points $z,w$,
  show that there exists $\varphi \in \Mult(\cH)$ with $\varphi(z) = 0$ and $\varphi(w) \neq 0$.
\end{proof}
Thus, $\Mult(\cH)$ separates the points of $X$ if and only if the kernel is irreducible in the strong
sense of \cite[Definition 7.1]{AM02}.

Suppose now that $\cH$ is a normalized complete Pick space which separates the points of $X$.
When investigating whether the corona theorem holds for $\Mult(\cH)$, we wish to exclude
constructions such as the restriction of the Hardy space $H^2$ to $\frac{1}{2} \ol{\bD}$,
which is really a space on $\bD$ in disguise. Thus, we will typically assume that $X$
is a maximal domain for $\cH$.

In this section, we are interested in the case when $X$ is a compact topological space
and the functions in $\cH$ are continuous on $X$.
Such spaces are easy to construct, as the following class of examples shows.

\begin{exa}
  \label{exa:H_s}
  Let $d \in \bN$ and let $\cH$ be a complete Pick space on $\bB_d$
  with reproducing kernel of the form
 \begin{equation*}
    k(z,w) = \sum_{n=0}^\infty a_n \langle z,w \rangle^n,
  \end{equation*}
  where $a_0 = 1$ and $a_n > 0$ for $n \ge 1$. If $\sum_{n=0}^\infty a_n < \infty$,
  but the power series $\sum_{n=0}^\infty a_n t^n$ has radius of convergence $1$,
  then $k$ extends to a continuous function on $\ol{\bB_d} \times \ol{\bB_d}$,
  and $\cH$ thus becomes a space of continuous functions on $\ol{\bB_d}$ in a natural way.
  It is not hard to see that $\ol{\bB_d}$ is a maximal domain for such a space
  (see, for example, \cite[Lemma 5.3 (b)]{Hartz15}). Moreover,
  since $a_1 > 0$, $\cH$ contains the coordinate functions, so that $\cH$ separates the points of $\ol{\bB_d}$.

  More concretely, for $s \in \bR$, let
  \begin{equation*}
    k(z,w) = \sum_{n=0}^\infty (n+1)^s (z \ol{w})^n \quad (z,w \in \bD)
  \end{equation*}
  and let $\cH_s$ be the reproducing kernel Hilbert space on $\bD$ with kernel $k$.
  This scale of spaces contains in particular the Bergman space ($s=1$), the Hardy space ($s=0$)
  and the Dirichlet space ($s=-1$). If $s \le 0$, then $\cH$ is a normalized complete Pick space (see,
  for example, \cite[Corollary 7.41]{AM02}). If $s < -1$, then $k$ satisfies the conditions
  in the preceding paragraph, hence $\cH_s$ becomes normalized complete Pick
  space on $\ol{\bD}$ and $\ol{\bD}$ is a maximal domain for $\cH$.
\end{exa}

Let $\cH$ be a normalized complete Pick space of continuous functions on a compact
set $X$ which separates the points of $X$ such that $X$ is a maximal domain for $\cH$.
Then the embedding of $X$ into $\cM(\Mult(\cH))$ via point evaluations is a homeomorphism
onto its image, hence $X$ can be identified with a compact
subset of $\cM(\Mult(\cH))$. Thus,
the corona theorem holds for $\Mult(\cH)$ if and only if $X = \cM(\Mult(\cH))$.
Moreover, a multiplier
$\varphi \in \Mult(\cH)$ is bounded below on $X$ if and only if it is non-vanishing.
Consequently, the one-function corona theorem for $\Mult(\cH)$ holds if and only if
every non-vanishing multiplier on $\cH$ is invertible.

We are now in the position to prove Proposition \ref{pra4}.

\begin{prop}
  \label{prop:compact_corona}
  Let $\cH$ be a
  normalized complete Pick space of continuous functions
  on a compact set $X$ which separates the points of $X$ such that $X$ is a maximal domain for $\cH$.
  Then the following are equivalent.
  \begin{enumerate}[label=\normalfont{(\roman*)}]
    \item $\Mult(\cH) = \cH$ as vector spaces.
    \item The corona theorem holds for $\Mult(\cH)$.
    \item The one-function corona theorem holds for $\Mult(\cH)$.
  \end{enumerate}
\end{prop}

\begin{proof}
  (i) $\Rightarrow$ (ii) Suppose that $\Mult(\cH) = \cH$ as vector spaces. Since the multiplier norm
  dominates the norm of $\cH$, an application of the open mapping theorem shows that
  these two norms are in fact equivalent. Thus, if $\rho$ is a character on $\Mult(\cH)$,
  then $\rho$ is a bounded functional on $\cH$ which is partially multiplicative.
  Since $X$ is a maximal domain for $\cH$, the functional $\rho$ equals evaluation
  at a point in $X$.

  (ii) $\Rightarrow$ (iii) follows from Gelfand theory.

  (iii) $\Rightarrow$ (i) Let $f \in \cH$. By Theorem \ref{thm:smirnov},
  there are $\varphi,\psi \in \Mult(\cH)$ with $\psi$ non-vanishing such that
  $f = \varphi / \psi$. The assumption (iii) implies that $1 / \psi \in \Mult(\cH)$,
  thus $f \in \Mult(\cH)$. The reverse inclusion always holds, hence
  $\cH = \Mult(\cH)$ as vector spaces.
\end{proof}

The spaces $\cH_s$ of Example \ref{exa:H_s}, where $s < -1$, satisfy
condition (i) of the preceding proposition (see Proposition 31 and Example 1 on page 99 in \cite{Shields74}).
Hence, $\cM(\Mult(\cH_s)) = \ol{\bD}$ (see also Corollary 1 on page 95 in \cite{Shields74}).
We now use an example of Salas \cite{Salas81}, which answered \cite[Question 15]{Shields74},
and Proposition \ref{prop:compact_corona} to exhibit a complete Pick space
on $\ol{\bD}$ for which the one-function corona theorem fails, thereby proving Theorem \ref{thm:corona_intro}.

\begin{thm}
  \label{thm:corona}
  There exists a complete Pick space $\cH$ on $\ol{\bD}$ with a reproducing kernel
  of the form
  \begin{equation*}
    k(z,w) = \sum_{n=0}^\infty a_n (z \ol{w})^n,
  \end{equation*}
  where $a_0 = 1, a_n > 0$ for all $n \in \bN$, $\lim_{n \to \infty} a_n / a_{n+1} = 1$ and $\sum_n a_n <\infty$
  such that $\ol{\bD}$ is a maximal domain for $\cH$ and such that
  the one-function corona theorem for $\Mult(\cH)$ fails.
\end{thm}

\begin{proof}
  In \cite{Salas81}, Salas constructs a weighted shift $T$ on $\ell^2$
  with weight sequence $(w_n)$ which satisfies
  \begin{enumerate}
    \item $w_n$ is decreasing and $\lim_{n \to \infty} w_n = 1$.
    \item $\sum_{n=0}^\infty \beta(n)^{-2} < \infty$, where $\beta(n) = \prod_{j=0}^{n-1} w_j$.
    \item $T$ is not strictly cyclic.
  \end{enumerate}
  The original definition of strict cyclicity can be found in \cite[Section 9]{Shields74};
we shall give an equivalent one.
Define a Hilbert space $\cH$ by
  \begin{equation*}
    \cH = \Big\{ f(z) = \sum_{n=0}^\infty \widehat f(n) z^n : ||f||^2 = \sum_{n=0}^\infty |\widehat f(n)|^2
  \beta(n)^2 < \infty \Big\}.
  \end{equation*}
Property (2) implies that $\cH$ is in fact
  a reproducing kernel Hilbert space on $\ol{\bD}$ whose reproducing kernel is given by
  \begin{equation*}
    k(z,w) = \sum_{n=0}^\infty a_n (z \ol{w})^n,
  \end{equation*}
  where $a_n = \beta(n)^{-2}$, see \cite[Section 6]{Shields74}.
Property (3) is equivalent to saying that $\Mult(\cH) \subsetneq \cH$,
see \cite[Proposition 31]{Shields74}.

We have $a_0 = 1$,
  so  $k$ is normalized at $0$.
  Moreover,
  \begin{equation*}
    \frac{a_n}{a_{n+1}} = \frac{\beta(n+1)^2}{\beta(n)^2} = w_n^2,
  \end{equation*}
  hence Property (1) implies that $a_n / a_{n+1}$ decreases to $1$. An application
  of a lemma of Kaluza (see Lemma 7.31 and 7.38 in \cite{AM02}) shows that $\cH$ is a
  complete Pick space.
  Moreover, we see that the radius of convergence of the power series
  $\sum_{n=0}^\infty a_n t^n$ is $1$, hence $\cH$ is a
  space of continuous functions on $\ol{\bD}$, which is a maximal domain for $\cH$,
  and $\cH$ separates the points of
  $\ol{\bD}$ (see Example \ref{exa:H_s}).
  Since $\Mult(\cH) \subsetneq \cH$,
  the implication (iii) $\Rightarrow$ (i) of Proposition \ref{prop:compact_corona} shows
  that the one-function corona theorem fails for $\Mult(\cH)$.
\end{proof}

We call a space $\cH$ on $\ol{\bD}$ as in the statement of Theorem \ref{thm:corona} a \emph{Salas space}.

\begin{rem}
  \label{rem:salas_properties}
  We can say slightly more about the maximal ideal space of the multiplier algebra of a Salas space $\cH$.
  It is easy to check (see, for example, \cite[Section 8]{Hartz15}) that there exists a continuous map
  \begin{equation*}
    \pi : \cM(\Mult(\cH)) \to \ol{\bD}, \quad \rho \mapsto \rho(z).
  \end{equation*}
  Since $\lim_{n \to \infty} a_n / a_{n+1} = 1$,
  it follows from \cite[Proposition 8.5]{Hartz15} that if $\lambda \in \bD$, then
  $\pi^{-1}(\lambda)$ is the singleton containing the character of evaluation at $\lambda$. In particular,
  every character on $\Mult(\cH)$ which is not given by evaluation at a point in $\ol{\bD}$
  is contained in $\pi^{-1}(\partial \bD)$.
  Since there exists at least one such character by Theorem \ref{thm:corona},
  and since $\cH$ is rotationally invariant, we deduce
  that $\pi^{-1}(\lambda)$ contains at least one such character for every $\lambda \in \partial \bD$.

  Moreover, since $\Mult(\cH)$ has no non-trivial idempotent elements, the \v{S}ilov idempotent theorem (see \cite[Theorem 3.5.1]{Kaniuth09}) shows that $\cM(\Mult(\cH))$ is connected.
\end{rem}

\begin{rem}
  Let $\cA$ be a unital Banach algebra of continuous functions on $\ol{\bD}$ such that the polynomials
  form a dense subspace of $\cA$ and such that the maximal ideal space of $\cA$ is equal to $\ol{\bD}$.
  In \cite[Section 2]{ENZ98}, the following problem is studied: given $\delta > 0$, does there exist
  a constant $C(\delta) > 0$ such that for all $f \in \cA$ with $||f||_{\cA} \le 1$ and
  $\inf_{z \in \ol{\bD}} |f(z)| \ge \delta$, the estimate
  \begin{equation*}
    ||f^{-1}||_{\cA} \le C(\delta)
  \end{equation*}
  holds? The authors of \cite{ENZ98} obtain a positive answer for rotationally
  invariant algebras $\cA$ which satisfy some additional assumptions.

  We can use a Salas space to obtain a rotationally invariant algebra for which the question above has
  a negative answer. To this end, let $\cH$ be a Salas space and let $A(\cH)$ denote the norm
  closure of the polynomials inside of $\Mult(\cH)$. The map $\pi$ of Remark \ref{rem:salas_properties}
  shows that the maximal ideal space of $A(\cH)$ is $\ol{\bD}$. By Theorem \ref{thm:corona},
  there exists a multiplier $\varphi$ in the unit ball of $\Mult(\cH)$
  such that $\delta = \inf_{z \in \ol{\bD}} |\varphi(z)| > 0$, but such that $\varphi$ is not invertible
  inside of $\Mult(\cH)$. For $r \in (0,1)$, define $\varphi_r(z) = \varphi(r z)$. Then each $\varphi_r$
  is analytic in an open neighborhood of $\ol{\bD}$, so since $\sigma_{A(\cH)}(z) = \ol{\bD}$, we conclude that
  $\varphi_r \in A(\cH)$ for all $r \in (0,1)$.
  Clearly, $|\varphi_r|$ is bounded below by $\delta$ for all $r \in (0,1)$, so
  each $\varphi_r$ is invertible inside of $A(\cH)$ by Gelfand theory. Moreover,
  a routine application of the Poisson kernel, combined with rotational invariance of $\cH$,
  shows that $||\varphi_r||_{A(\cH)} \le ||\varphi||_{\Mult(\cH)} \le 1$ for all $r \in (0,1)$.

  We claim that $||\varphi_r^{-1}||_{A(\cH)}$ is not bounded as $r \to 1$. Indeed, suppose otherwise.
  Then by weak-$*$ compactness of the closed unit ball of $\Mult(\cH)$, the net $(\varphi_{r}^{-1})_{r < 1}$
  has a weak-$*$ cluster point $\psi \in \Mult(\cH)$. In particular, $\varphi_{r}^{-1}$ converges to $\psi$
  pointwise on $\ol{\bD}$, so that $\psi  = \varphi^{-1}$, contradicting the fact that $\varphi$ is not invertible
  inside of $\Mult(\cH)$.
\end{rem}

If $\cH$ is a Salas space, then the polynomials are not norm dense in $\Mult(\cH)$, since the corona
theorem fails for $\Mult(\cH)$. However, besides this and Remark \ref{rem:salas_properties}, we know very little
about the size of $\Mult(\cH)$ or of its maximal ideal space.

\begin{quest}
  Let $\cH$ be a Salas space on $\ol{\bD}$.
  \begin{enumerate}
    \item Is $\Mult(\cH)$ separable?
    \item Is $\cM(\Mult(\cH))$ metrizable?
    \item Is the cardinality of $\cM(\Mult(\cH))$ equal to that of the continuum?
  \end{enumerate}
\end{quest}

Observe that a positive answer to any of these questions implies
a positive answer to the questions below it.

Recall that $H^2_d$ denotes the Drury-Arveson space on $\bB_d$, the
open unit ball in a Hilbert space of dimension $d$.
If $d = \aleph_0$, we simply write $\bB_\infty$ and $H^2_\infty$.
In the case $d < \infty$, Costea, Sawyer and Wick \cite{CSW11} showed that the corona theorem
holds for $\Mult(H^2_d)$. Fiang and Xia \cite{FX13} provide
a more elementary proof of the one-function corona theorem in this case; an even shorter
proof was found by Richter and Sunkes \cite{RS16}. But none of these proofs
extend to infinite $d$ in a straightforward manner.
Let $\cH$ be a Salas space.
It follows from Theorem \ref{thm:AM} of Agler and McCarthy that $\Mult(\cH)$
can be identified with
\begin{equation*}
  \{\varphi \big|_V: \varphi \in \Mult(H^2_\infty) \}
\end{equation*}
for a set $V \subset \bB_\infty$ (indeed, we can choose $V= b(\ol{\bD})$, where $b$
is the map from Theorem \ref{thm:AM}).
We are not aware of an argument which would show that the failure of the corona
theorem for $\Mult(\cH)$ implies the failure of the corona theorem for $\Mult(H^2_\infty)$.
We therefore ask:

\begin{quest}
  Does the corona theorem hold for $\Mult(H^2_\infty)$? Does the one-function corona theorem hold for $\Mult(H^2_\infty)$?
\end{quest}

\bibliographystyle{amsplain}
\bibliography{literature}

\end{document}